\documentclass[1p]{elsarticle}

\usepackage{lineno,hyperref}

\usepackage{amssymb}
\usepackage{eucal}
\newcommand{\R}{{\Bbb R}}
\newcommand{\Z}{{\Bbb Z}}

\newcommand{\C}{{\Bbb C}}

\usepackage{amsmath}

\newtheorem{thm}{Theorem}
\newtheorem{lemma}[thm]{Lemma}
\newtheorem{corollary}[thm]{Corollary}
\newtheorem{proposition}[thm]{Proposition}

\newtheorem{remark}[thm]{Remark}

\newproof{proof}{Proof}

\begin{document}

\begin{frontmatter}

\title{Stability of semi-wavefronts for delayed reaction-diffusion equations}

\author[a]{Abraham Solar}
\address[a]{Instituto de F\'isica, Pontificia Universidad Cat\'olica de Chile, Casilla 306, Santiago, Chile
\\ {\rm E-mail: asolar@fis.uc.cl}}

\begin{abstract} This paper deals with the asymptotic behavior of solutions to the delayed
monostable  equation: $(*)$ $u_{t}(t,x) = u_{xx}(t,x) - u(t,x) + g(u(t-h,x)),$ $x \in \mathbb{R},\ t >0,$ where $h>0$ and the reaction term $g: \mathbb{R}_+ \to \mathbb{R}_+$ has exactly two fixed points (zero  and $\kappa >0$).  Under certain condition on the derivative of $g$ at $\kappa$, the global stability of fast wavefronts is proved. Also, the stability of the $leading \ edge$ of semi-wavefronts for $(*)$ with $g$ satisfying $g(u)\leq g'(0)u, u\in\R_+,$ is established.\end{abstract}


\end{frontmatter}


\section{Main results and discussion}

In this work, the main object of studies is:  
\begin{eqnarray}\label{ie}
u_{t}(t,x) &=& u_{xx}(t,x) - u(t,x) + g(u(t-h,x)),\quad x \in \R,\ t >0,
\end{eqnarray}
where $h>0$ and the birth function $g:\R_+\rightarrow\R_+$ is Lipschitz continuous with constant $L_g$ and has only two fix points, 0 and $\kappa>0$ .
\vspace{1.5mm}

It is well known that one of key elements determining the dynamics in (\ref{ie}) is  its semi-wavefront solutions. By definition, these are bounded positive solutions of (\ref{ie}) in the form $u(t,x)=\psi_c(x+ct)$, where $c>0$ is called the speed of propagation and the wave profile $\psi_c:\R\rightarrow\R$ satisfies   $\psi_c(-\infty)=0$. When $\psi_c(+\infty)=\kappa$ we say that $\psi_c$ is a wavefront.

 The studies of stability of wavefronts for delayed monostable models were initialed by Schaaf  \cite{sch}  who considered a quite general scalar reaction-diffusion equation. In particular,  Schaaf established that the localization of the spectrum of the linearization of (\ref{ie}) along $u(t,x)=\psi(x+ct)$ depends continuously on delay. It should be noted that one of the main  difficulties appearing in the study of wavefronts for delayed equations is that the  associated semi-flow is not monotone in general. In order to deal with this obstacle, a $quasimonotonicity$ condition is usually assumed:   it requires from the nonlinearity  monotonicity  in the delayed argument (as it was done in  \cite{sch}). After the seminal work by Schaaf, a series of other studies  appeared, in some of them the $subtangency$ condition (i.e., the inequality $g(u)\leq g'(0)u$, $u\in\R_+$) was instrumental for the stability analysis  (in any case, a number of applied models satisfies this condition).  Among these studies, we would like to distinguish the important contribution \cite{MOZ}  by Mei {\it et al}  
where, assuming  the quasi-monotonicity condition, the authors proved  the global stability of monotone wavefronts (minimal as well as non minimal ones, see \cite[Theorem 2.2]{MOZ}). Then, we can to mention a work of Sh-L Wu {\it et al} \cite{WZL} where for some type of subtangential non monotone $g$ ('crossing monostable' case), the local stability of fast wavefronts is obtained; it is worth mentioning that in this case, wavefronts, in general, are no longer monotone (e.g., see \cite{TTT} and \cite{FGT}). Still in subtangential  case, Chi-Kun Lin $et\ al$ \cite{LLLM} prove, for unimodal $g$(this is, supposing  $g$ has a single local maximum), the local stability of non minimal wavefronts(whether monotone or non monotone).  Next, under similar conditions and $g$ satisfying $|g'(u)|<1$ in some neighborhood of $\kappa$, Solar and Trofimchuk have established the global stability of non minimal wavefronts (whether monotone or non monotone)\cite[Corollary 3]{STR1}. Using that $|g'(\kappa)|<1$, Chern  $et\ al$ \cite[Theorem 2.3]{CMYZ}   have demonstrated the local stability of critical wavefronts (whether monotone or non monotone) for a class of initial data sufficiently flat at $\infty$.

 Now, equations with non subtangential nonlinearities have recently attracted a lot of interest because of their connection to the so-called  Allee Effect in the population dynamics \cite{BGHR,BB,Den,RGHK}, which means that the per capita growth rate $g(u)/u$ for low densities $u$ is also low. In this direction, by assuming quasi-monotonicity condition, it has been possible to establish the stability of pushed wavefronts (global in a certain sense) \cite[Theorem 1.5]{STR} as well as that of non minimal wavefronts \cite[Theorem 1]{STR1}. 
 
Hence, in the above mentioned works, we can find  the stability results  for equation (\ref{ie}) only when $g$ meets either monotonicity or subtangency condition. It is important to mention that without assuming any of these two conditions, even the problem of existence of wavefronts is not completely solved (e.g.,  see \cite[Corollary 4]{TTT}: of course,  in the available literature there are some partial results on the existence  of semi-wavefronts for certain subclasses of equations, e.g., see \cite[Theorem 2.4]{TT}). 
In the case of non minimal wavefronts, one of the main assertions of this paper(Corollary \ref{gs})  extends stability results  from  the aforementioned works. Actually,   without assuming monotonicity or subtangency property of $g$, we prove the global stability of wavefronts with a speed greater than certain speed $c(L_g)$. This work  is also one of the first studies on the stability of proper semi-wavefronts. In the case when $c(L_g )$ coincides with the minimum speed, the stability of {\it leading edge} all semi-wavefronts is obtained. We also present a new result concerning  the uniqueness of semi-wavefronts: Corollary  \ref{uni}  below complements Theorem 7 of \cite{AG}   proved for local reaction-diffusion equation (\ref{ie}).    

\noindent In order to obtain these stability results we  study the decay of solutions of the constant coefficient linear equation with delay, 

\begin{equation}\label{lde}
u_t(t,x)=u_{xx}(t,x)+mu_x(t,x)+pu(t,x)+qu(t-h,x+d) \quad x\in\R, t>0,
\end{equation}
where the parameters $m, p, q$ y $d$ are real numbers.
 
For an initial datum $u_0\in C([-h,0], L^{1}(\R))$, let us denote $C_{u_0}:=\sup_{s\in[-h,0]}||u(s,\cdot)||_{L^1}$.
\begin{thm}\label{lee}  
Suppose that $-p\geq q\geq 0$. Let $\gamma$ be the only real and non-positive solution of the following equation:

\begin{eqnarray}\label{P}
\gamma-p= qe^{-h\gamma}.
\end{eqnarray}

\noindent If the initial datum $u_0$ belong to $C ([-h,0]; L^1(\R))$ then the solution $u(t,x)$ of (\ref{lde}) satisfies the estimate:
\begin{equation}\label{deca}
\sup_{x\in\R}|u(t,x)| <A_0 \frac{e^{\gamma t}}{\sqrt{t}}	\quad\quad \hbox{for all} \quad t>\frac{h}{2},
\end{equation}
where $A_0=C_{u_0}/2\sqrt{1+h(\gamma-p)}$,
\end{thm}
We note that in the critical case $p  = q$ (which implies $\gamma = 0$) an exponential estimate is no longer available. In some cases, it can be established that the decay is not faster than that given by (\ref{deca}).   For instance, if $d = 0$ for the evolution equation (\ref{lde}), the behavior of the solutions in the $L^1(\R)$ phase space with an appropriate weight can be specified. In fact, we obtain the exact behavior which is embodied in Theorem below 

\begin{thm}[Asymptotic behavior]\label{lee1}

Let us consider (\ref{lde}) with $d = 0$.  Let $u(t, x)$  the solution generated by the initial data $u(s,\cdot)=e^{\sigma s}u_0$  where $u_0$  is such that $e^{\frac{m}{2}\cdot}u_0\in L^1(\R)$ and  $\sigma$  is the only real solution of 

\begin{eqnarray}
qe^{-\sigma h}=\sigma+\frac{m^2}{4}-p, 
\end{eqnarray}
then
\begin{eqnarray}
\lim_{t\rightarrow\infty}\sqrt{t}e^{-\sigma t}u(t,x+o(\sqrt{t}))=\frac{\sqrt{1+hqe^{-\sigma h}}}{2\sqrt{\pi}}e^{-\frac{m}{2}x}\int_{\R}e^{\frac{m}{2}y}u_0(y)dy,
\end{eqnarray}
for all $x\in\R$. 
\end{thm}
Now, for the study of the stability of semi-wavefronts with speed $c$ , the following equation should be considered

\begin{eqnarray}\label{nle2}
v_t(t,z)=v_{zz}(t,z)-cv_z(t,z)-v(t,z)+g(v(t-h,z-ch)) \quad t>0, z\in\R.
\end{eqnarray}

\vspace{2mm}

\noindent For  fixed $L > 1$ and a parameter $c$, let us define the real function $E_c(\lambda)=-\lambda^{2}+c\lambda+1-Le^{-\lambda ch}$ and 

$$
c(L)=\inf \{c>0: E_c(\lambda_0)\geq 0\ \hbox{for some}\ \lambda_0\in\R \}.
$$

Let us denote by $\lambda_1(c)\leq\lambda_2(c)$ the two solutions of $E_c(\lambda)=0$ for $c\geq c(L)$. For $c\geq c(L_g)$ let us fix  $\lambda_c\in[\lambda_1(c),\lambda_2(c)]$ and let us denote by $\xi_c(z):=e^{-\lambda_c z}$.  Now, the first main result of this article can be set out.

\begin{thm}[Stability with weight]\label{st}
Let $c\geq c(L_g)$.  Let $v_0(s, z)$ and $\psi_0 (s, z)$ be two initials data to (\ref{nle2}) such that
 \begin{eqnarray}\label{ic}
 u_0(s,z):=\xi_c (z)|v_0(s,z)-\psi_0(s,z)|\in C([-h,0], L^{1}(\R)) 
 \end{eqnarray}
 Assume $u(t, z)$  satisfies (\ref{lde})  with initial data $u_0(s, z)$ and with parameters  $m=m(\lambda_c)=2\lambda_c -c$, $p=p(\lambda_c)=\lambda_c^{2}-c\lambda_c-1$, $q=q(\lambda_c)=L_ge^{-\lambda_c ch}$ and $d=-ch$.
  Then:
\begin{eqnarray}\label{rc}
 \xi_c(z)|v(t,z)-\psi(t,z)|\leq u(t,z)\quad \hbox{for all}\quad t\geq-h, z\in\R,
\end{eqnarray}
in particular
\begin{eqnarray}\label{rc1}
|v(t,z)-\psi(t,z)|\leq A_0 \xi_c(-z)\frac{e^{\gamma t}}{\sqrt{t}}\quad \hbox{for all}\quad t>\frac{h}{2}, z\in\R
\end{eqnarray}
where $\gamma=\gamma(\lambda_c)$ is defined by $(\ref{P})$
\end{thm}

\begin{corollary}[Uniqueness]\label{uni}
If $\phi_1(z)$ and $\phi_2(z)$ are two semi-wavefronts with speed  $c\geq c(L_g)$ satisfing (\ref{ic}), then $\phi_1=\phi_2$.

\end{corollary}
\begin{remark}
In the inequality (\ref{rc}), instead of using the general functions $u(t,x)$ defined by Theorem \ref{st}, one could use the particular functions $u(t,x)=Be^{\gamma t+\lambda x}$, where  $B\in\R_+$ and $\gamma$ and $\lambda$ satisfy $$qe^{-\gamma h}e^{\lambda d}=-\lambda^2-m\lambda-p+\gamma,$$ however if ones does so one miss the important critical case, $c=c(L_g)$.

\end{remark}

\begin{remark}
When $\psi(t, z)$ is a wavefront, we have the stability of the wavefront on the sets $(-\infty, N ], N \in \R.$ This result is comparable to the one obtained by Uchiyama in \cite[ Theorem 4.1]{UC}.
\end{remark}
\vspace{2mm}

Theorem \ref{st} shows the stability of the $leading \ edge$ of semi-wavefronts with the correction given by the weight $\xi_c$ due to the instability of the equilibrium 0. It should be noted that this result does not take into account the stability of the positive equilibrium so that the semi-wavefronts, for example, could be asymptotically periodic at $\infty$ \cite[Theorem 3]{TTT}. The Corollary 1.4 essentially refers to the fact that semi-wavefronts are equal (up to translation) if they have the same asymptotic behavior in $-\infty$, i.e., the condition (\ref{ic}).
In the case when $L_g  = g'(0)$, under a rescaling of the variables, the
Theorem \ref{st} includes well-known models such as Nicholson's, described by,

 \begin{eqnarray} \label{nich} 
u_t(t,x)=u_{xx}(t,x)-\delta u(t,x)+pu(t-h,x)e^{-u(t-h,x)}\quad t>0, x\in\R,
\end{eqnarray}
or the Mackey-Glass model given by
\begin{eqnarray} \label{MG} 
u_t(t,x)=u_{xx}(t,x)-d u(t,x)+\frac{ab^n u(t-h,x)}{b^n+u^n(t-h,x)}\quad t>0, x\in\R.
\end{eqnarray}

In this case, we have the stability (and uniqueness  up  to  translation) of the leading edge of semi-wavefronts with speed c of (\ref{nich}) and (\ref{MG}) for all $c \geq c_*.$   Here,  $c_*$   is the minimum speed for which semi-wavefronts exist (see  \cite[Theorem 4.5]{TT} ).  The semi-wavefronts  of Theorem \ref{st} could show a type of $connective \ instability$ due to the positive equilibrium (e.g., see \cite{BS}); however, by controlling the size of the slope of $g$ at the positive equilibrium, the stability of the complete semi-wavefront can be obtained. For this, it is necessary to make some additional hypotheses, such as the following monostability condition

\vspace{2mm}

\noindent  {\rm \bf(M)}  
The function $g:\R_+\rightarrow\R_+ $ is such that the equation $g(x)= x$ has exactly two solutions: $0$ and
$\kappa>0$. Moreover, $g$  is $C^1$-smooth in some
$\delta_0$-neighborhood of the equilibria  where $g'(0) >1>g'(\kappa).$ In addition,
there are $C >0,\ \theta \in (0,1],$ such that   $
\left|g'(u)- g'(0)\right| +|g'(\kappa) - g'(\kappa-u)| \leq Cu^\theta $ for $u\in
(0,\delta_0].$
\vspace{4mm}

We note that for $g$ satisfying {\bf(M)}, there are real numbers $0 < \zeta_1\leq \zeta_2$
such that

\begin{itemize}
\item[(B1)] $g([\zeta_1,\zeta_2])\subset[\zeta_1,\zeta_2]$ and $g([0,\zeta_1])\subset [0,\zeta_2]$;
\item[(B2)] $\min_{\zeta\in[\zeta_1,\zeta_2]}g(\zeta)=g(\zeta_1)$;
\item[(B3)] $g(x)>x$ for $x\in[0,\zeta_1]$ and $1<g'(0)\leq g^*_+:=\sup_{s\geq 0}g(s)/s<\infty$;
\item[(B4)] In $[0,\zeta_2]$, the equation $g(x)=x$ has exactly two solutions 0 and $\kappa$.
\end{itemize}

Due to this and to \cite[Theorem 4.5]{TT} we obtain the following
\begin{proposition}[Existence of semi-wavefronts] \label{ew}
Let $g$ satisfy {\bf (M)}. 
Then, for each  $c > c^*_+:=  c(g^*_+)\geq c(L_g)$ equation  (\ref{ie})  has  semi-wavefronts 
with speed $c$.  Moreover, if  $0 < \zeta_1\leq\zeta_2$  meet (B1)-(B4), then, each   semi-wavefront $\phi_c$  satisfy:

\begin{eqnarray*}
\zeta_1\leq\phi_c(z)\leq\zeta_2\quad\forall z\in\R.
\end{eqnarray*}
\end{proposition}
\vspace{3mm}



Let us introduce the following notation.   If $I\subset\R_+ = Dom(g)$, let us denote by

$$
L_g(I):=\sup_{x\neq y; x,y\in I}\frac{|g(x)-g(y)|}{|x-y|},
$$
and for  $b \in\R$, let us denote by $\eta_b(z)=\min\{1,e^{\lambda_c(z-b)}\}$.
With these notations, the second main result of this paper can be established

\vspace{2mm}

\begin{thm}[Global stability]\label{sc}
Let $c> c(L_g)$ and $\bar{g}$ a nondecresing function meeting {\bf(M)} with equilibrium $K$ such that $\bar{g}(u)\geq g(u)$ for all $u\in\R_+$ such that $L_{\bar{g}}\leq L_g$. We denote by $m_K=\min_{u\in[\kappa,K]}g(u)$ and $\mathcal{I}_K:=[m_K,K]$ and we suppose that $L_g(\mathcal{I}_K)<1$. Then, if for certain $q>0$ and $b\in\R$ we have

\begin{eqnarray}\label{inqu2}
|v_0(s,z)-\psi_0(s,z)|\leq q\eta_{b}(z)\quad \hbox{for all}\quad (s,z)\in[-h,0]\times\R,
\end{eqnarray} 
then, there exists $C=C(\bar{g},m_K,b)>0$ such that for any $\gamma_0\geq 0$ satisfying
\begin{eqnarray}\label{gamma}
-\lambda_c ^2+c\lambda_c+1\geq\gamma_0+ L_ge^{\gamma_0 h}e^{-\lambda_c ch}\ \hbox{and}\quad L(I)\leq e^{-\gamma_0 h}(1-\gamma_0),
\end{eqnarray}
we have
\begin{eqnarray}\label{Inqu2}
|v(t,z)-\psi(t,z)|\leq Cqe^{-\gamma_0 t}\quad \forall (t,z)\in[-h,\infty)\times\R.
\end{eqnarray}
Moreover, for each semi-wavefront $\phi_c$ of the equation (\ref{ie}) we have: $\phi_c\in\mathcal{I}_K$ 
\end{thm}
\vspace{3mm}

If we take $\bar{g}(u)=\max_{s\in[0,u]}g(s)$ then we have that $K=M_g:=max_{s\in[0,\kappa]}g(s)$. Now, by writing $m_g=\min_{u\in[\kappa,M_g]}g(u)$ and $\mathcal{I}_K=I_g:=[m_g, M_g]$ the following corollary is obtained

\begin{corollary}\label{gs}

Let $g$ satisfy {\bf(M)} such that $L_g (I_g ) < 1$. If $\psi_c$  is a wavefront with speed $c \geq c(L_g )$, then $\psi_c$  is globally stable in the sense of Theorem \ref{sc}
 \end{corollary}

 In principle, the statement of Corollory \ref{gs} should say semi-wavefronts, but due to  \cite[Lemma 4.1]{TT} and the fact $g$ is a contraction in the interval $I_g$, it follows that a semi-wavefront which stays in $I_g$  necessarily is a wavefront.  Corollary \ref{gs} generalizes the results for wavefronts in the monotone, subtangential and unimodal cases, e.g., \cite{MOZ,STR1}. In non subtangential and monotone cases, Corollary \ref{gs} is an improvement, in terms of the globality of the disturbance, of \cite[Theorem 1.5]{STR} for wavefronts with a speed greater than $c(L_g )$ and it also gives us an exponential convergence rate for these waves. In this sense, exponential stability for pushed wavefronts was not established in \cite{STR} but a recent work by S-L Wu {\it et al} \cite{WNH} did it.
 
This paper is organized as follows: linear theorems (Theorem \ref{lee} and Theorem \ref{lee1}) are proven in Section 2. Finally,  results on the stability of semi-wavefronts are proven in Section 3.

\section{Proof of Linear Theorems}
In order to demonstrate both Theorem \ref{lee} and Theorem \ref{lee1}, the following two lemmas will be needed; the first one is an abstract version of the Halanay type inequalities \cite{H}

\begin{lemma}[Halanay Type Inequality] \label{halanay}

Let $\sigma,k\in \C$ and $X$ a complex Banach space.   If for $h>0$, $r\in C([-h,\infty),X)$ is a function satisfiying:

\begin{eqnarray*}
r_t(t)= \sigma r(t)+kr(t-h) \quad \hbox{a.e.,}
\end{eqnarray*}

\noindent then:
 \begin{eqnarray}\label{h1}
|r(t)|\leq (\sup_{s\in[-h,0]} |r(s)|)e^{\min\{0,-\lambda\} h} e^{\lambda t}\quad\hbox{for all}\quad t>-h,
 \end{eqnarray}
 where $\lambda$ is the only real root of the equation :
 \begin{eqnarray} \label{elam}
 \lambda=Re(\sigma)+|k|e^{-\lambda h}.
\end{eqnarray}
 Moreover 
  
 \begin{enumerate}
 \item[(i)] $\lambda \leq 0 \iff -Re(\sigma)\geq |k| $
 
 \item[(ii)] $\lambda = 0 \iff -Re(\sigma)= |k| $
 \end{enumerate}
\end{lemma}
\begin{proof} It is clear that:
$$
 \frac{d}{dt}(r(t)e^{-\sigma t})=ke^{-\sigma t}r(t-h)\quad a.e.
$$

\noindent and from here, it is obtained that $|r(t)|$  meets the following inequality:

\begin{eqnarray}\label{II}
x(t)\leq |k|\int_{0}^{t} e^{Re(\sigma)(t-s)}x(s-h)ds+x(0)\quad \hbox{for all}\quad t>-h
\end{eqnarray}

\noindent We note that for $A\in\R$ the function $e_A (t)=Ae^{\lambda t}$ meets (\ref{II}) with equality. Now, for $A:= \sup_{s\in[-h,0]}|r(s)|e^{\min\{0,-\lambda\} h}$ the function $\delta(t)=|r(t)|-e_{A}$ satisfies (\ref{II}) for $t\in[0,h]$ and therefore $\delta(t)\leq 0$ for all $t\in[0,h]$. Similarly, it is concluded that $\delta(t)\leq 0$ for the intervals $[h,2h],[2h,3h]...$  This proves (\ref{h1}).

Let us prove (i). If $-Re(\sigma)\geq |k|$ then: $\lambda\leq|k|(e^{-h\lambda}-1)$ which necessarily implies that $\lambda\leq 0$. Otherwise,
if $\lambda\leq 0$ let us suppose that $-Re(\sigma)<b$, then $\lambda>|k|(e^{-h\lambda}-1)$ which is a contradiction.

In order to prove (ii) let us note that since the derivative of $f(\lambda):=\lambda-Re(\sigma)-|k|e^{-h\lambda}$ is always positive, then $f(\lambda)$ has at most one zero. 
If $Re(\sigma)=|k|$ then $\lambda=0$ is the only solution. From (\ref{elam}), if $\lambda=0$ then $a=b$. $\quad\square$
\end{proof}
\vspace{3mm}

\noindent Now, let us consider the equation (\ref{elam}) in the form
\begin{eqnarray}\label{lamb}
\lambda(\zeta)=-\zeta^2 +p+qe^{-h\lambda(\zeta)},
\end{eqnarray}
where $\zeta\in\R$, and we estimate the function $\lambda(\zeta)$. 

For $\epsilon_h=\frac{1}{1+h(\gamma-p)}$ we define the function 
$$\alpha_{h}(\zeta):=-\frac{1}{h}\log(1+h\epsilon_h \zeta^2).$$

\begin{lemma}\label{loge}
 We suppose that $q\geq 0$. If $\gamma\leq 0$ satisfies (\ref{P}),
then
\begin{eqnarray}\label{gauss}
-\epsilon_h\zeta^2+\gamma\leq\lambda(\zeta)\leq \alpha_{\epsilon}(\zeta)+\gamma \quad \hbox{for all}\ \zeta\in\R.
\end{eqnarray}
Moreover
\begin{eqnarray}\label{log}
\lim_{\zeta\rightarrow\infty}\frac{\lambda(\zeta)}{\log(\zeta)}=-\frac{2}{h}.
\end{eqnarray}
\end{lemma}

\begin{remark}
The function $\alpha_{h}$   is a generalization of the function  $\alpha_0 (\zeta):=-\zeta^2$,  which corresponds to the case $q=0$. So,  in  this  case,  for each $\zeta\in\R$ 
$\lim_{h\rightarrow 0}\alpha_h(\zeta)=-\zeta^{2}$ is obtained. Furthermore, due to the fact that $\lim_{\zeta\rightarrow 0 }\alpha_h(\zeta)/-\epsilon_h\zeta^2=1$ and (\ref{log}) the estimations in (\ref{gauss}) are optimal.
\end{remark}

\begin{proof} Let us denote $\beta(\zeta)=\lambda(\zeta)-\alpha(\zeta)-\gamma$. Then $\beta(\zeta)$ satisfies the following equation 
$$
 \beta(\zeta)=-\zeta^{2}+\frac{1}{h}\log(1+h\epsilon_h\zeta^2)-\gamma+p+qe^{-h\gamma}(1+h\epsilon_h\zeta^2)e^{-h\beta(\zeta)}.
$$
 From Lemma \ref{halanay}  we have that $\beta(\zeta)\leq 0$ if and only if:
 \begin{eqnarray}\label{desl}
\zeta^{2}-\frac{1}{h}\log(1+h\epsilon_h\zeta^2)+\gamma-p\geq qe^{-h\gamma}(1+h\epsilon_h\zeta^2).
\end{eqnarray}
Now, using $\log(1+x)\leq x$, fo rall $x\geq 0$, in order to obtain (\ref{desl}) it is enough to have
\begin{eqnarray*}
\zeta^2-\epsilon_{h}\zeta^2+\gamma-p\geq qe^{-h\gamma}(1+h\epsilon_h\zeta^2)\quad\hbox{for all}\quad\zeta\in\R\\
\iff (1-\epsilon_{h}-qh\epsilon_h e^{-h\gamma})\zeta^2+\gamma-p- qe^{-h\gamma}\geq 0\quad\hbox{for all}\quad\zeta\in\R\\
\iff 1-\epsilon_{h}-qh\epsilon_h e^{-h\gamma}= 0 \quad \wedge \quad \gamma-p- qe^{-h\gamma}= 0.
\end{eqnarray*} 
This proves (\ref{gauss}).

Now, let us prove (\ref{log}).   We note that $\lambda :  \R_+\rightarrow \lambda(\R_+ )$ is invertible with inverse  $r$ given by the formula $r(\lambda) = \sqrt{qe^{-h\lambda}-\lambda+p}$.  Therefore,  $\lambda(\zeta)$ is differentiable for all positive $\zeta$.  By calculating its derivative, the result is:

\begin{eqnarray}\label{dlb}
\lambda'(\zeta)= -\frac{2\zeta}{1+hqe^{-h\lambda(\zeta)}}.
\end{eqnarray}   
However (\ref{gauss}) tell us that $\lim_{\zeta\rightarrow\infty}\lambda(\zeta)=-\infty$ and from (\ref{dlb}) we have that :
$$ 
\lim_{\zeta\rightarrow\infty} \frac{e^{-\lambda(\zeta)h}}{\zeta^2}=\lim_{\zeta\rightarrow\infty}\frac{-h\lambda'(\zeta)e^{-\lambda(\zeta)h}}{2\zeta}=\lim_{\zeta\rightarrow\infty}\frac{he^{-\lambda(\zeta)h}}{1+hqe^{-\lambda(\zeta)h}}=\frac{1}{q}.$$
This, together with (\ref{dlb}), allow us to obtain
\begin{eqnarray*}
\lim_{\zeta\rightarrow\infty}\frac{\lambda(\zeta)}{\log(\zeta)}=\lim_{\zeta\rightarrow\infty}\zeta\lambda'(\zeta)=-2\frac{\zeta^2}{1+hqe^{-\lambda(\zeta)h}}=-\frac{2}{h}. \quad\square
\end{eqnarray*}
\end{proof}

\vspace{2mm}

\paragraph{\underline{Proof of Theorem \ref{lee}}} 
Due to the hypothesis made about $u_0$, when applying the Fourier transform, we have

\begin{eqnarray*}
\hat{u_t}(t,\zeta)=\sigma(\zeta) \hat{u}(t,\zeta)+k(\zeta)\hat{u}(t-h,\zeta)\quad \hbox{for all}\quad t>0,
\end{eqnarray*}
 where $\sigma(\zeta)=-\zeta^2+im\zeta+p$ and $k(\zeta)=qe^{-id\zeta}.$
 
 Since $-Re(\sigma(\zeta))\geq |k(z)|$, by Lemma \ref{halanay} we obtain that:
 $$
 |\hat{u}(t,\zeta)|\leq C_{u_0} e^{\lambda(\zeta)t}.
 $$
 If $t>\frac{h}{2}$ then by the Inversion Theorem (because $u(t,\cdot)\in C^1(\R)$) and Lemma 
\ref{loge}, we have

$$
|u(t,x)|\leq \frac{1}{2\pi}\int_{\R} |\hat{u}(t,\zeta)|d\zeta\leq \frac{C_{u_0}}{2\pi}\int_{\R} e^{\lambda(\zeta)t}d\zeta\leq \frac{C_{u_0}}{2\pi}e^{\gamma t}\int_{\R} \frac{d\zeta}{(1+\epsilon\zeta^2)^{\frac{t}{h}}}.
$$
Moreover,  by Bernoulli's inequality, we conclude that

$$
\int_{\R} \frac{d\zeta}{(1+\epsilon\zeta^2)^{\frac{t}{h}}}\leq \int_{\R}\frac{d\zeta}{1+\frac{t\epsilon}{h}\zeta^2}=\frac{1}{\sqrt{t}}[\sqrt{\frac{h}{\epsilon}}\int_{\R}\frac{d\zeta}{1+\zeta^2}]=\frac{1}{\sqrt{t}}\sqrt{\frac{h}{\epsilon}}\pi.\quad\square$$

\paragraph{\underline{Proof of Theorem \ref{lee1}}}

If we make the change $v(t,x)=e^{\frac{m}{2}x}u(t,x)$, then 
$v(t, x)$ solves
\begin{eqnarray}\label{ldet}
v_t(t,x)=v_{xx}(t,x)+(p-\frac{m^2}{4})v(t,x) +q(t-h,x).
\end{eqnarray}

\noindent Applying the Fourier transform to (\ref{ldet}), we have

\begin{eqnarray}\label{eh}
\hat{v}_t(t,z)=(-z^2+p-\frac{m^2}{4})\hat{v}(t,z)+q\hat{v}(t-h,z).
\end{eqnarray}

Let us note that due to $q \geq 0$, we have that (\ref{eh}) satisfies the Comparison
Principle; that is, if $v_0$  and $w_0$  are two initial data (\ref{eh}), satisfying  
$$ 
 v_0(s)\leq w_0(s)\quad \hbox{for all} \ s\in[-h,0],
$$
then 
$$
v(t)\leq w(t)\quad\hbox{for all} \ t>-h.
$$
Let us denote by $Re(\hat{v}(t,z))=v_1(t,z), Im(\hat{v}(t,z))=v_2(t,z)$ and $e_{A}(t,z)=Ae^{\lambda(z)t}$, where $\lambda(z)$ satisfies
\begin{eqnarray}\label{lbd}
\lambda(z)=-z^2+p-\frac{m^2}{4}+qe^{-\lambda(z)h}.
\end{eqnarray}
Let us note that $e_A(t,z)$ satisfies (\ref{eh}) for all $A\in\C$.
Also, let us denote that
$$
m_i(z)=\min_{s\in[-h,0]}(v_i(s,z) e^{-\lambda(z)s})\quad\hbox{y}\quad M_i(z)= \max_{s\in[-h,0]}(v_i(s,z) e^{-\lambda(z)s})\quad i=1, 2
$$

\noindent then we have that 
$$
e_{m_i}(s,z)\leq v_i(s,z)\leq e_{M_i}(s,z)\quad \hbox{for all} \ (s,z)\in[-h,0]\times\R; i=1,2.
$$
\noindent By the comparison principle applied to real and imaginary part in (\ref{eh}), we have that  
   \begin{eqnarray}
e_{m_i}(t,z)\leq v_i(t,z)\leq e_{M_i}(t,z)\quad\hbox{for all}\quad (t,z)\in[-h,\infty)\times\R; i=1,2
\end{eqnarray}
or
 \begin{eqnarray}\label{inql}
m_i(z)e^{\lambda(z)t}\leq v_i(t,z)\leq M_i(z)e^{\lambda(z)t}\quad\hbox{for all}\quad (t,z)\in[-h,\infty)\times\R; i=1,2
\end{eqnarray}

\noindent Now, by the Fourier inversion formula , we have that

\begin{eqnarray}\label{fr}
v(t,x)=\frac{1}{\sqrt{t}}\int_{\R}e^{\frac{xy}{\sqrt{t}}}\hat{v}(t,y/\sqrt{t})dy.
\end{eqnarray}

\noindent However, by (\ref{gauss}) we have 
\begin{eqnarray}
\lim_{t\rightarrow\infty}t[\lambda(y/\sqrt{t})-\sigma]=-\frac{y^2}{1+hqe^{-\sigma h}}.
\end{eqnarray}

\noindent and due to $v(s, \cdot) \in L^1 (\R)$ by the Lebesgue's dominated convergence theorem

$$
\lim_{t\rightarrow\infty}M_1(y/\sqrt{t})=\lim_{t\rightarrow\infty}m_1(y/\sqrt{t})=\int_{\R}e^{\frac{m}{2}x}v(s,x)dx
$$
and
$$
\lim_{t\rightarrow\infty}M_2(y/\sqrt{t})=\lim_{t\rightarrow\infty}m_2(y/\sqrt{t})=0.
$$
Therefore by (\ref{inql}) 
\begin{eqnarray}\label{limv}
\lim_{t\rightarrow\infty}\hat{v}(y/\sqrt{t})=e^{-\frac{y^2}{1+hqe^{-\sigma h}}}\int_{\R}e^{\frac{m}{2}x}v(s,x)dx 
\end{eqnarray}

\noindent However, by (\ref{h1}) there exists $C (p, q, m) > 0$ such that 

$$
|\hat{v}(t,y/\sqrt{t})|\leq C\sup_{s\in[-h,0]}|\hat{v}(s,y/\sqrt{t})| e^{\lambda(y/\sqrt{t})t}
$$
but by (\ref{gauss}) and Bernoulli's Inequlity
\begin{eqnarray}\label{dom}
|\hat{v}(t,y/\sqrt{t})|\leq \frac{Ce^{|\sigma|h}||e^{\frac{m}{2}\cdot}u_0(\cdot)||_{L^1(\R)}}{1+\epsilon_h y^2}\quad\hbox{for all}\quad t>0.
\end{eqnarray}

\noindent Finally, by (\ref{fr}), (\ref{dom}), Lebesgue's dominated convergence theorem and (\ref{limv}), the result obtained.$\quad\square$

\section{Proof of results of stability of semi-wavefronts}

\paragraph{\underline{Proof of Theorem \ref{st}}}  
For a solution $w(t, z)$ of (\ref{nle2}), let us denote the function $\tilde{w}(t,z)=\xi_c(z)w(t,z)$ which satisfies

$$
\tilde{w}_t(t,z)=\tilde{w}_{zz}(t,z)+m\tilde{w}_z(t,z)+p\tilde{w}(t,z)+\xi_c(z)g(\xi_c(-z+ch)\tilde{w}(t-h,z-ch)).
$$

\noindent We consider the linear operator $$\mathcal{L}\delta(t,z):=\delta_{zz}(t,z)+m\delta_z(t,z)+p\delta(t,z)-\delta_t(t,z).$$

\noindent  If $\delta_{\pm}(t,z):=\pm[\tilde{v}(t,z)-\tilde{\psi}(t,z)]-u(t,z)$, then by (\ref{ic}): $\delta_{\pm}(s,z)\leq 0$ for $(s,z)\in[-h,0]\times\R$.  For $(t,z)\in[0,h]\times\R$  by (\ref{nle2}) and (\ref{ic}) we have
\begin{eqnarray*}
\mathcal{L}\delta_{\pm}(t,z)&=&\mp\xi(z)[g(\xi(-z+ch)\tilde{\psi}(t-h,z-ch))-g(\xi(-z+ch)\tilde{v}(t-h,z-ch))]-\mathcal{L}u(t,z)\\
&\geq&-L_ge^{-\lambda ch}|\tilde{v}(t-h,z-ch)-\tilde{\psi}(t-h,z-ch)|-\mathcal{L}u(t,z)\\
&\geq& - L_ge^{-\lambda ch}u(t-h,z-ch)-\mathcal{L}u(t,z)= 0.
\end{eqnarray*}

\noindent Now, by the Phragm\`en-Lindel\"of principle from \cite{PW}[Chapter 3, Theorem 1], we have $\delta_{\pm}(t,z)\leq 0$ for $(t,z)\in[0,h]\times\R$.  The argument is repeated for intervals $[h, 2h], [2h, 3h]...$ to conclude (\ref{rc}).        
Finally,(\ref{rc1}) is obtained using the Theorem \ref{lee}.$\quad\square$

\vspace{4mm}

\begin{thm}\label{gest} 
Let $v(t,z)$ and $\psi(t,z)$ be solutions of equation (\ref{nle2}) for $c\geq c(L_g)$. Assume that for some compact interval $I \subset \R$, such that $L_g (I ) < 1,$ we have

\begin{eqnarray}\label{inv0}
\psi(t,z),v(t,z)\in I \quad \hbox{for all} \ (t,z)\in[-h,\infty)\times[b-ch,\infty),
\end{eqnarray}

  \noindent then, if for some $q > 0$, we have

\begin{eqnarray}\label{inqu3}
|v_0(s,z)-\psi_0(s,z)|\leq q\eta_b(z)\quad \hbox{for all}\ (s,z)\in[-h,0]\times\R,
\end{eqnarray} 

\noindent then, for any $\gamma_0 \geq 0$ such that

\begin{eqnarray}\label{gamma}
-\lambda_c ^2+c\lambda_c+1\geq\gamma_0+ L_ge^{\gamma_0 h}e^{-\lambda_c ch}\ \hbox{and}\quad L(I)\leq e^{-\gamma_0 h}(1-\gamma_0),
\end{eqnarray}
we have that 
\begin{eqnarray}\label{Inqu3}
|v(t,z)-\psi(t,z)|\leq qe^{-\gamma_0 t}\eta_b(z)\quad \hbox{for all}\ (t,z)\in[-h,\infty)\times\R.
\end{eqnarray}

\end{thm}

\begin{proof}
 Let us denote by $\eta(t,z)=qe^{-\gamma_0 t}\eta_b(z)$ and write the operator
$$
\mathcal{L}_0\delta(t,z):=\delta_{zz}(t,z)-c\delta_z(t,z)-\delta(t,z)-\delta_t(t,z).
$$

\noindent Notice that if $\delta_{\pm}(t,z):=\pm[v(t,z)-\psi(t,z)]-\eta(t,z)$ then $\delta_{\pm}(s,z)\leq 0$ for $(s,z)\in[-h,0]\times\R$. Now, for $(t,z)\in[0,h]\times(-\infty,b]$ due to (\ref{nle2}), (\ref{inqu2}) and (\ref{gamma}) we have that
\begin{eqnarray*}
\mathcal{L}_0\delta_{\pm}(t,z)&=&\pm[-g(v(t-h,z-ch))+g(\psi(t-h,z-ch))]-\mathcal{L}_0\eta(t,z)\\
 &\geq& qe^{-\gamma_0 t+\lambda_c z}[-L_ge^{\gamma_0 h}e^{-\lambda ch}-(\lambda_c^{2}-c\lambda_c-1+\gamma_0)]\geq 0.
\end{eqnarray*}
 Similarly,  if $(t,z)\in[0,h]\times[b,\infty)$ we obtain:
  \begin{eqnarray*}
\mathcal{L}_0\delta_{\pm}(t,z)&=&\pm[-g(v(t-h,z-ch))+g(\psi(t-h,z-ch))]-\mathcal{L}_0\eta(t,z)\\
 &\geq& qe^{-\gamma_0 t}[-L(I)e^{\gamma_0 h}\eta(z-ch)-(-1+\gamma_0)]\\
&\geq& qe^{-\gamma_0 t}[-L(I)e^{\gamma_0 h}+1-\gamma_0]\geq 0.
\end{eqnarray*}

\noindent Now, as in the proof of the \cite[Lemma 1]{STR1}, due to 
 \begin{eqnarray}
\frac{\partial \delta_{\pm}(t,b+)}{\partial z}- \frac{\partial \delta_{\pm}(t,
b-)}{\partial z} > 0, \label{di}
\end{eqnarray}
we have that $\delta_{\pm}(t,z) \leq 0$ for all $t \in [0,h], \ z \in
\R$. Indeed, otherwise there exists $r_0> 0$ such that
$\delta(t,z)$ restricted to any rectangle $\Pi_r= [-r,r]\times
[0,h]$ with $r>r_0$,   reaches its maximum positive value $M_r >0$
at some point $(t',z') \in \Pi_r$.

We claim  that $(t',z')$ belongs to the parabolic boundary
$\partial \Pi_r$ of $\Pi_r$. Indeed, suppose on the contrary, that
$\delta(t,z)$ reaches its maximum positive value at some point
$(t',z')$ of $\Pi_r\setminus \partial \Pi_r$. Then clearly $z'
\not=z_*$ because of (\ref{di}). Suppose, for instance that $z' >
z_*$. Then $\delta(t,z)$ considered on the subrectangle $\Pi=
[z_*,r]\times [0,h]$ reaches its maximum positive value $M_r$ at the
point $(t',z') \in \Pi \setminus \partial \Pi$.  Then the
classical results \cite[Chapter 3, Theorems 5,7]{PW} show that
$\delta(t,z) \equiv M_r >0$ in $\Pi$, a contradiction.

Hence, the usual maximum principle holds for each $\Pi_r, \ r \geq
r_0,$ so that we can appeal to the proof of the
Phragm\`en-Lindel\"of principle from \cite{PW} (see Theorem 10 in
Chapter 3 of this book), in order to conclude that  $\delta(t,z)
\leq 0$ for all  $t \in [0,h], \ z \in \R$.

We can again repeat the above argument on the intervals $[h,2h],$ $[2h, 3h], \dots$ establishing that the inequality $w_-(t,z) \leq w(t,z)\leq w_+(t,z),$  $z\in \R,$ holds for all $t \geq -h$.  $\quad\square$
\end{proof}


\begin{remark}\label{invBC}
   We can generalize the function $\eta_b (z)$ for $b = \infty$ and, thus, have $\eta_{\infty}(z)=\xi_c(-z)$. In this proof, it was not necessary to use the condition (\ref{inv0}) for $z \leq b$ so by replacing $\xi_c(-z)$ by $\eta_b(z)$ it can be concluded that (\ref{inqu3}) implies (\ref{Inqu3}).
\end{remark}

\vspace{3mm}

\begin{corollary}[Local stability]\label{algbr}
Let us suppose that there exist $M, b \in\R$ and $l_0  > 0$, such that:
\begin{eqnarray}\label{inv1}
\psi(t,z)\in[M-l_0,M+l_0]\quad\hbox{for all}\quad (t,z)\in[-h,\infty)\times[b-ch,\infty),
\end{eqnarray}
and that for some $l_1  > l_0$  the initial data satisfies
\begin{eqnarray}\label{Inv}
|v_0(s,z)-\psi_0(s,z)|<(l_1-l_0)\eta_b(z)\quad\hbox{for all} \quad (s,z)\in[-h,0]\times\R.
\end{eqnarray}

\noindent If we denote by $I_1 := [M - l_1, M + l_1]$ and $L(I_1) < 1$, then, there exists $\gamma_0\geq 0$
satisfying (\ref{gamma}), such that

\begin{eqnarray}\label{inv2}
|v(t,z)-\psi(t,z)|\leq (l_1-l_0)e^{-\gamma_0 t}\eta_b(z)\quad \hbox{for all} \ (t,z)\in[-h,\infty)\times\R.
\end{eqnarray}
\end{corollary}

\begin{proof} 
Clearly, $\psi(t,z)\in I_1, \hbox{for all} (t,z)\in[-h,\infty)\times[b-ch,\infty).$ Now if we suppose that the inequality in (\ref{Inv}) is satisfied in an interval of the form $I_k=[h(k-2),h(k-1)]$ with $k\in\Z_+$, then: $\psi(t,z)\in I_1, \hbox{for all} \ (t,z)\in I_k $ . Then arguing as in the proof of Theorem \ref{gest}, this implies that $\mathcal{L}_0\delta_{\pm}(t,z)\leq 0$ for $(t,z)\in I_{k+1}$, where $\delta_{\pm}(t,z)=\pm[v(t,z)-(\psi(t,z)\pm (l_1-l_0)e^{-\gamma_0 t}\eta_b(z))]$ and by \cite[Lemma 1]{STR1} we concluded that: $\delta_{\pm}(t,z)\leq 0, \hbox{for all}\ (t,z)\in I_{k+1}$. It is therefore sufficient to suppose (\ref{Inv})  to conclude  (\ref{inv2})  step by step.$\quad\square$
\end{proof}

\noindent To prove Theorem \ref{sc}, we will use the following lemma
\begin{lemma}\label{lem}
Let us suppose that functions  $g_1,g_2: D\subset\R_+\rightarrow\R_+$ satisfy: $g_1(u)\leq g_2(u)$ for all $u\in D$. Let $v_1(t, z), v_2(t, z) : [h, \infty)\times \R\rightarrow  D$ solutions to (\ref{nle2}),  with $g = g_1$  and $g = g_2$,  respectively, such that: $v_1(s, z) \leq v_2 (s, z)$  for  $(s, z) \in  [h, 0] \times\R$.   If $g_1$  or $g_2$  is a nondecreasing function, then, we have : $v_1(t, z) \leq v_2(t, z)$ for all $(t, z) \in \R_+ \times \R.$
\end{lemma}
 
 \begin{proof} We take $\delta(t,z)=v_1(t,z)-v_2(t,z)$. Let us note that if  $(t,z)\in[0,h]\times\R$ then 
 $$
 \mathcal{L}_0\delta(t,z)=g_2(v_2(t-h,z-ch))-g_1(v_1(t-h,z-ch))\geq 0,
  $$
  because if $g_2$  is a nondecreasing function  we have that
 
 $$g_2(v_2(t-h,z-ch))-g_1(v_1(t-h,z-ch))\geq g_2(v_1(t-h,z-ch))-g_1(v_1(t-h,z-ch))\geq 0,$$
 
 \noindent or if $g_1$  is a nondecreasing function,  we have

  $$g_2(v_2(t-h,z-ch))-g_1(v_1(t-h,z-ch))\geq g_2(v_2(t-h,z-ch))-g_1(v_2(t-h,z-ch))\geq 0$$ 
  
\noindent  Now, as $\delta(t,z)\leq 0$ for all $(t,z)\in[-h,0]\times\R$ the Phragm\`en-Lindel\"of principle from  \cite{PW}[Chapter 3, Theorem 10] implies that $\delta(t,z)\leq 0$ for $(t,z)\in[0,h]\times\R$. The argument is repeated for intervals $[h,2h],[2h,3h]...\quad\square$
  \end{proof}

\vspace{3mm}
\paragraph{\underline{Proof Theorem \ref{sc}}}

\noindent Let us take  $\epsilon>0$ such that $L(\mathcal{I}_{\epsilon})<1$, where $\mathcal{I}_{\epsilon}:=[m_K-\epsilon,K+\epsilon]\subset\R_+.$Then, there is an increasing function $\bar{g}_{\epsilon}$ satisfying {\bf(M)} with positive equilibrium $\kappa_+\in(K,K+\epsilon)$, $c(L_{\bar{g}_{\epsilon}})\leq c( L_g)$ and $g\leq\bar{g}_{\epsilon}$.  Furthermore, there is also an increasing $g$ function meeting {\bf(M)} with positive equilibrium $\kappa_-\in(m_K-\epsilon,m_K)$ and $c(L_{\underline{g}})\leq c(L_g)$ such that: $\underline{g}(x)\leq g(x)$ for $x\in[0,K+\epsilon]$.


  Now, if $\bar{v}(t,z)$ is the solution of (\ref{nle2}) replacing $g$ by $\bar{g}_\epsilon$ with initial data $v_0(s,z)$ and $c>c(L_g)$ then by Lemma \ref{lem}, Proposition \ref{ew} and \cite[Theorem 1]{STR1} there is wavefront $\phi_c^{\bar{g}}$ and a number $T>0$ such that
\begin{eqnarray} \label{doms}  
v(t,z)\leq \bar{v}(t,z)\leq \phi_c^{\bar{g}_\epsilon}(z)\leq K+\epsilon\quad \hbox{for all}\quad (t,z)\in[T,\infty)^2.
\end{eqnarray}

\noindent We denote by  $\underline{v}(t,z)$ the solution of (\ref{nle2}) replacing $g$ by $\underline{g}$ with initial data $\underline{v}_0(s,z)=v(s+T+h,z)$. Then, for $c>c(L_g)$ by Lemma \ref{lem}(with $D=[0,K+\epsilon]$), Proposition \ref{ew} and \cite[Theorem 1]{STR1}  there is a wavefront $\phi_c^{\underline{g}}$ and $T_0>0$ such that
\begin{eqnarray} \label{domi}
m_K-\epsilon\leq\phi^{\underline{g}}_c(z)\leq\underline{v}(t,z)\leq v(t,z)\quad \hbox{for all}\quad (t,z)\in[T_0,\infty)^2.
\end{eqnarray}

\noindent Finally, for $t_0:=\max\{T,T_0\}$ we define $\tilde{v}(t,z):=v(t+t_0+h,z)$ and $\tilde{\psi}(t,z):=\psi(t+t_0+h,z)$. Then, these function satisfy (\ref{inv0}) with $b=t_0+ch$ and $I=\mathcal{I}_{\epsilon}$. Besides, by Remark \ref{invBC} we have that $\tilde{v}_0(s,z)$ and $\tilde{\psi}_0(s,z)$ satisfy the condition (\ref{inqu2}). This latter, along with  (\ref{domi}) allow us to apply Theorem \ref{gest} to conclude  (\ref{Inqu2}) with $C:=\max_{z\in\R}\eta_{b}(z)/\eta_{t_0+ch}(z)\quad\square$

 \section*{Acknowledgments} This work was supported by FONDECYT (Chile) through the Postdoctoral  Fondecyt  2016 program with project number 3160473.  The author is very grateful to Dr.  Sergei Trofimchuk for his important comments on this work.

\end{document}